\documentclass[a4paper,reqno]{amsart}
\usepackage{graphicx} 
\usepackage{amsmath,amsthm,amssymb,amscd,amsthm,mathtools,mathrsfs}
\usepackage{listings}
\usepackage{color}   
\usepackage{hyperref}
\usepackage{physics}

\hypersetup{
    colorlinks=false, 
    linktoc=all,     
}
\lstdefinestyle{mystyle}{
    basicstyle=\ttfamily\footnotesize,
    breakatwhitespace=false,         
    breaklines=true,                                     
    keepspaces=true,                                     
    numbersep=5pt,                  
    showspaces=false,                
    showstringspaces=false,
    showtabs=false,                  
    tabsize=2
}

\newtheorem{conjecture}{Conjecture}[section]
\newtheorem{theorem}{Theorem}[section]
\newtheorem{lemma}{Lemma}[section]

\newtheorem{subprop}{Proposition}[section]

\begin{document}
\title{Explicit estimates of the weighted sum $S(x)=\sum_{n \leq x} (-2)^{\Omega(n)} \log\bigl(\frac{x}{n}\bigr).$}
\author[Riddhi Manna] {Riddhi Manna}
\address{
    School of Science, University of New South Wales (Canberra), Northcott Drive, Campbell, ACT 2600, Australia
}
\email{r.manna@unsw.edu.au}
\begin{abstract}
We study the oscillatory arithmetic function $(-2)^{\Omega(n)}$, where $\Omega(n)$ counts the number of prime factors of $n$, with multiplicity. Sun conjectured a bound on its partial sums $W(x) = \sum_{n \leq x}^{} (-2)^{\Omega(n)}$ as $|W(x)| < x$ for all $x \geq 3078$. In this direction, we obtain new bounds for its logarithmically weighted average
\begin{equation*}
    S(x)=\sum_{n \leq x} (-2)^{\Omega(n)} \log\biggl(\frac{x}{n}\biggr).
\end{equation*}
Using complex-analytic methods such as the log-weighted Perron's formula, we computed the bound $|S(x)| \leq 1.6x$.
\end{abstract}
\maketitle
\section{Introduction}
In \cite[Conjecture 1.1]{sun}, Sun proposes the following elegant conjecture on the sum
\begin{equation*}
    W(x) = \sum_{n \leq x}^{} (-2)^{\Omega(n)}. 
\end{equation*}
\begin{conjecture} \label{sun}
For all $x \geq 3078$, 
    \[
    |W(x)| < x. 
    \]
\end{conjecture}
Sun's conjecture has been verified to be true by Mossinghoff and Trudgian \cite{oscillations} for all $x \leq 2.5 \cdot 10^{14}$. In a recent paper by Johnston, Leong and Tudzi \cite{Johnston2024NewBA} this conjecture is shown to be true up to the order of $x$. They proved the following theorem.
\begin{subprop} \label{sun bound}
There exists a constant $C > 0$ such that for all sufficiently large $x$,
\begin{equation*}
    |W(x)| < Cx.
\end{equation*}
That is, $W(x) = O(x)$. Moreover, one can take $C = 2260$ for all $x \geq 1$.
\end{subprop}
$W(x)$ has jump discontinuities of size $x$ whenever $x$ is a power of $2$. Therefore there are infinitely many values of $x$ for which
\begin{equation*}
    |W(x)|\geq\frac{x}{2}.
\end{equation*}
This is because there exists a power of $2$ between $\frac{x}{2}$ and $x$ which will produce a jump of magnitude at least $\frac{x}{2}$.
Grosswald \cite[Theorem 2]{Grosswald} had previously worked on the related sum $\sum_{n \leq x}^{} 2^{\Omega(n)}$. He showed the following theorem.
\begin{subprop}
For some absolute constants $c_1$ and $c_2$,
\begin{equation}\label{grosswald}
    \sum_{n \leq x}^{} 2^{\Omega(n)} = c_1 x \log^2x + c_2 x \log x + O(x).
\end{equation}
\end{subprop}
This shows that there is a $\log^2x$ cancellation when we replace the $2$ in the sum by a $-2$.\\

The Conjecture \ref{sun} was proposed by Sun in the context of other historical conjectures, such as the conjectures about the sum
\begin{equation*}
    L(x) = \sum_{n \leq x}^{} (-1)^{\Omega(n)}.
\end{equation*}
P\'olya \cite{Pólya1919} showed that if the sign of the sum $L(x)$ is eventually constant, then the Riemann hypothesis would follow. He reported that after $L(1)=1$, for $x$ up to approximately $1500$, $L(x)$ was not positive. Later, Ingham \cite{Ingham1942} in his 1942 paper noted that if $L(x)$ is eventually of constant sign, then in addition to the Riemann hypothesis, the zeroes of $\zeta(s)$ would all be simple, and there would exist infinitely many linear dependencies over $\mathbb{Z}$ of the imaginary parts of the zeroes of $\zeta(s)$ in the upper half plane. Similarly, Tura\'n \cite{Turan1948} hypothesized that if the weighted sum $T(x) = \sum_{n \leq x}^{} \frac{(-1)^{\Omega(n)}}{n}$ is of constant sign for sufficiently large $x$, then it implies the Riemann hypothesis. He noted that $T(x)$ is never negative over $2\leq x\leq1000$. The other results about the simplicity and linear dependance of the zeros of $\zeta(s)$ also follows from this. Despite the numerical evidence in favour of both these conjectures, surprisingly in 1958, Haselgrove \cite{Haselgrove} proved that $L(x)$ and $T(x)$ both change sign infinitely often.\\

Another conjecture with a similar connection to the Riemann hypothesis is Merten's conjecture. We define the M\"obius function in the following way.
\begin{equation*}
    \mu(n)=
\begin{cases}
    & (-1)^{\Omega(n)}, \quad \text{if n is square-free}\\
    & 0, \quad \text{otherwise}
\end{cases}
\end{equation*}
Merten's conjecture is the statement that for all $n \geq 1$,
\begin{equation*}
    \bigg|\sum_{n \leq x}^{}\mu(n)\bigg| \leq \sqrt{x}.
\end{equation*}
It is well known that Merten's conjecture implies the Riemann hypothesis. In his 1942 paper, Ingham showed that Merten's conjecture implies that there are infinitely many linear dependencies over the rationals among the ordinates of
the zeros of the zeta function on the critical line in the upper half-plane. There was a large amount of computational evidence in favour of this conjecture, but it was disproved in 1985 by Andrew Odlyzko and Herman te Riele \cite{OdlyzkoRiele+1985+138+160}. In fact, $M(x)$ is conjectured to grow much faster than $\sqrt{x}$.\\

 Due to this long history of similar conjectures being disproved despite strong computational evidence, it could be expected that Sun's conjecture is also not true.
In this direction, we estimate $S(x)$, a smooth version of the sum, which is
\begin{equation*}
    S(x) = \sum_{n \leq x} (-2)^{\Omega(n)} \log\biggl(\frac{x}{n}\biggr).
\end{equation*}
\begin{theorem}\label{S(x) convergent sum}
    Let $z_k$ be
    \begin{equation*}
        z_k = 1 - \frac{i(2k+1)\pi}{\log 2}.
    \end{equation*}
    Also, we have the odd term Dirichlet series
    \begin{equation*}
        H(s) = \sum_{\substack{n=1\\n  \text{ odd}}}^{\infty} \frac{(-2)^{\Omega(n)}}{n^s}.
    \end{equation*}
    Then for $T= {(\log x)}^{1+\delta}$ and $\delta > 0$ we have
    \begin{align*}
        S(x) &= \frac{1}{\log2}\sum_{k \in \mathbb{Z}} \frac{H(z_k)x^{z_k}}{{z_k}^2} + O\biggl(\frac{x(\log T)^{\frac43}(\log \log T)^{\frac23}}{T^2}\biggr)\\&+O\biggl(\frac{x\log^2x}{T^2}\biggr)+O\biggl(\frac{x(\log T)^{\frac43}(\log \log T)^{\frac23}}{x^{c'(\log T)^{-\frac23}(\log\log T)^{-\frac13}}}\biggr). 
    \end{align*}
\end{theorem}
The odd term Dirichlet series $H(s)$ can be written in terms of $\zeta(s)$ as
\begin{equation*}
    H(s) = \frac{G(s)}{\zeta^2(s)}.
\end{equation*}
Where $G(s)$ converges absolutely in the half-plane $\mathfrak{R}(s)> \log2/\log 3$. Using bounds on $G(s)$ and $\zeta(s)$, we can compute an upper bound for $S(x)$.
\begin{theorem}\label{main theorem}
For all sufficiently large $x$, we have the bound
\begin{equation*}
    |S(x)| \leq 1.5478x.
\end{equation*}
\end{theorem}
Applying Abel's summation formula to $S(x)$ we get
\begin{equation*}
    |S(x)| = \Bigg|\int_{1}^{x}\frac{W(t)}{t}\,\text{d}t\Bigg| \leq 1.5478x.
\end{equation*}
This implies that while we do not get a point-wise bound on $W(x)$ via the bound on $S(x)$, we can show that on average $\frac{W(x)}{x}$ is bounded by a very small constant. To study the behaviour of $S(x)$, we plot $|S(x)|/x$ for integer values in the range $1 \leq x \leq 10^6$ in Figure \ref{Figure 1} below. Except for a few small values of $x$, for these integer values, the function  $S(x)/x$ is contained within the range $(-0.12,0.12)$. There are a few jump discontinuities at multiples of powers of $2$.
\begin{figure}[htp]
    \centering
    \noindent
    {\includegraphics[width=\linewidth]{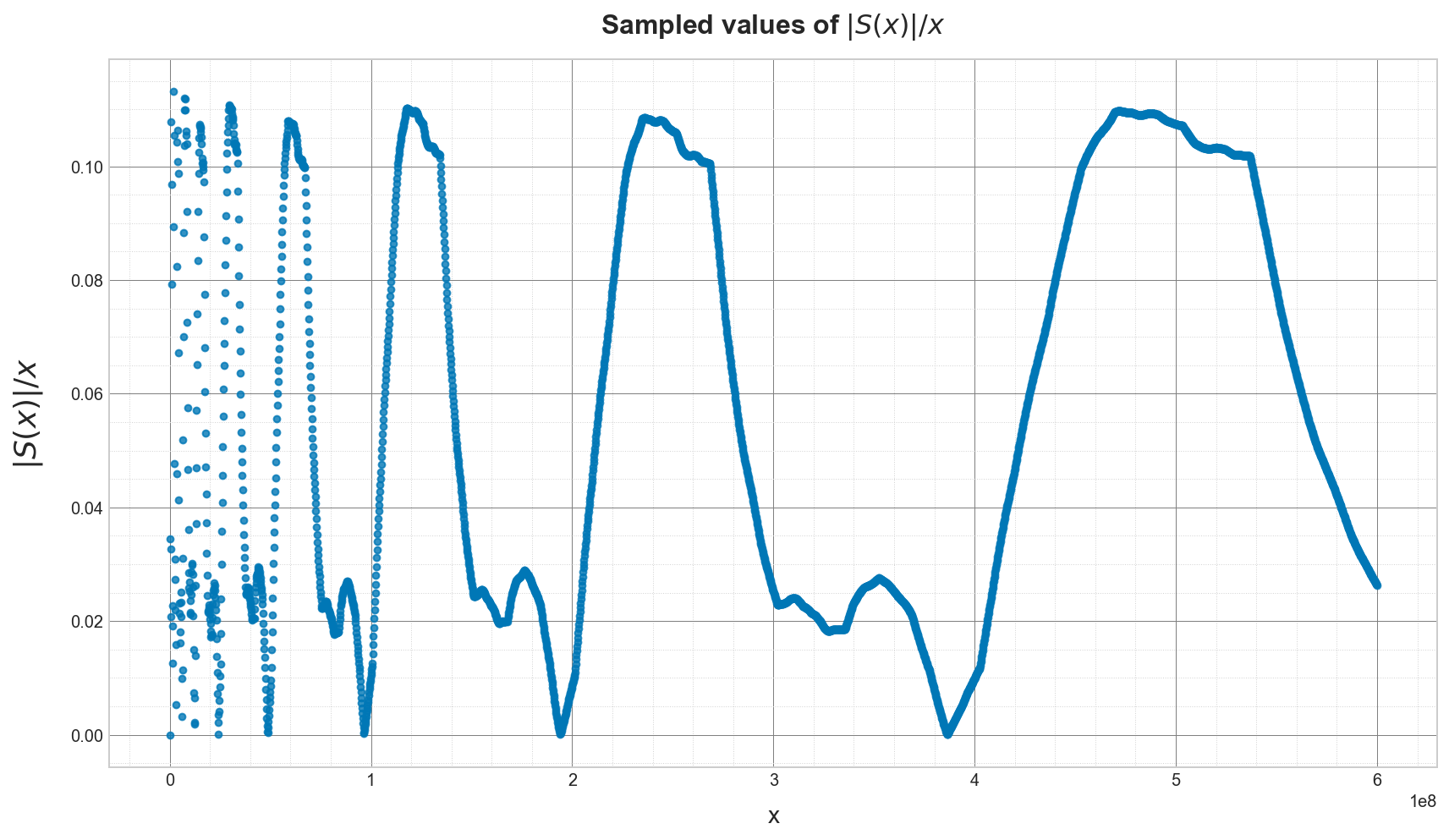}}
    \label{Figure 1}
\end{figure}

In Section \ref{perron section}, we first discuss the Dirichlet series associated with $S(x)$ and its various properties. In the following section, Section \ref{proof section}, we give the proof of Theorem \ref{S(x) convergent sum} via an application of Perron's formula and the residue theorem.
\section{Preliminary results}\label{perron section}
\subsection{Perron's formula}
The main tool used in approximating $W(x)$ is a weighted version of Perron's classical formula. Perron's formula relates the summatory function of an arithmetic function to a contour integral. The contour integral can then be evaluated using techniques from complex analysis. We use the following log-weighted Perron's formula to smooth the sum for easier analysis. 
\begin{lemma}\label{perron lemma}
    Let $f(s)$ be a Dirichlet series with $\sigma_a$ the abscissa of absolute convergence and given by
    \begin{equation*}
        f(s) = \sum_{n=1}^{\infty} \frac{a_n}{n^s}. 
    \end{equation*}
    Then for ${\sigma_{0}} > \sigma_a$, $T \geq 1$, and $x \geq 2$ with $x = N + \frac{1}{2}$ for some $N \in \mathbb{Z}$, we have
\begin{equation} \label{perron}
    \sum_{n\leq x}^{} a_n \log\Bigl(\frac{x}{n}\Bigr)= \frac{1}{2 \pi i} \int_{{\sigma_{0}} - iT}^{{\sigma_{0}} + iT} f(s) \frac{x^s}{s^2} ds + O\Biggl(\frac{1}{T}\sum_{n=1}^{\infty} {\Bigl(\frac{x}{n}\Bigr)}^{{\sigma_{0}}} |a_n| \min\Biggl(1, \frac{1}{T\big|\log \frac{x}{n}\big|}\Biggr)\Biggr).
\end{equation}
\end{lemma}
\begin{proof}
    We first assume $0 < x <1$ and ${\sigma_{0}} > 0$ and estimate the integral 
    \begin{equation*}
        \int_{{\sigma_{0}}-iT}^{{\sigma_{0}}+iT} \frac{x^s}{s^2} ds.
    \end{equation*}
     Consider the counterclockwise contour $K_U$ given by the vertices ${\sigma_{0}}-iT$, ${\sigma_{0}}+iT$, $U+iT$ and $U-iT$ for some $U > 0$. Also, we take $U$ such that $U > {\sigma_{0}}$. By Cauchy's theorem,
    \begin{equation}
        \frac{1}{2 \pi i} \int_{K_U} \frac{x^s}{s^2} = 0.
    \end{equation}
    Now we can write
    \begin{equation*}
        \Bigg[\int_{{\sigma_{0}}+iT}^{{\sigma_{0}}-iT} + \int_{{\sigma_{0}}-iT}^{U-iT} +\int_{U-iT}^{U+iT} +\int_{U+iT}^{{\sigma_{0}}+iT} \Bigg] \frac{x^s}{s^2} ds = 0.
    \end{equation*}
    Rearranging we get,
    \begin{equation*}
        \int_{{\sigma_{0}}-iT}^{{\sigma_{0}}+iT} \frac{x^s}{s^2} ds = \Bigg[\int_{{\sigma_{0}}-iT}^{U-iT} +\int_{U-iT}^{U+iT} -\int_{{\sigma_{0}}+iT}^{U+iT}\Bigg] \frac{x^s}{s^2} ds.
    \end{equation*}
    We estimate the three integrals on the right-hand side of the above equation. Considering the third integral on the right, we have 
    \begin{equation*}
        \Bigg|\frac{1}{2\pi i}\int_{{\sigma_{0}}+iT}^{U+iT} \frac{x^s}{s^2} ds\Bigg| \leq \frac{1}{2\pi T^2}\int_{{\sigma_{0}}}^{U} x^{\sigma_{0}} d{\sigma_{0}} 
    \end{equation*}
    Since $0 < x <1$, as $U \to \infty$, this integral is bounded by
    \begin{equation*}
        \frac{x^{\sigma_{0}}}{2 \pi T^2|\log x|}.
    \end{equation*}
    This same bound holds for the other horizontal integral. For the vertical integral
    \begin{equation*}
        \Bigg|\frac{1}{2\pi i}\int_{U-iT}^{U+iT} \frac{x^s}{s^2} ds\Bigg| \leq \frac{x^U}{2 \pi U^2} \int_{-T}^{T} dt = \frac{x^UT}{\pi U^2}. 
    \end{equation*}
    Now, taking $U \to \infty$, as $0 < x <1$, this term goes to zero. So for $0<x<1$ and ${\sigma_{0}} > 0$
    \begin{equation}\label{0<x<1}
        \Bigg|\frac{1}{2 \pi i}\int_{{\sigma_{0}}-iT}^{{\sigma_{0}}+iT} \frac{x^s}{s^2} ds \Bigg| \leq \frac{x^{\sigma_{0}}}{2\pi T^2 |\log x|}.
    \end{equation}
    The above inequality will give us one of the two terms in the error term
    \begin{equation}\label{perron error}
        O\Biggl(\frac{1}{T}\sum_{n=1}^{\infty} {\Bigl(\frac{x}{n}\Bigr)}^{{\sigma_{0}}} |a_n| \min\Biggl(1, \frac{1}{T\big|\log \frac{x}{n}\big|}\Biggr)\Biggr),
    \end{equation}
    among which we consider the minimum. For the other term in \eqref{perron error}, we 
    need the following inequality 
    \begin{equation*}
         \Bigg|\frac{1}{2\pi i}\int_{{\sigma_{0}}-iT}^{{\sigma_{0}}+iT} \frac{x^s}{s^2} ds\Bigg| < \frac{x^{\sigma_{0}}}{T}.
    \end{equation*}
    We consider the semi-circular contour of radius $({\sigma_{0}}^2+T^2)^{\frac12}$ centred at the origin. This passes through ${\sigma_{0}} - iT$ and ${\sigma_{0}} + iT$. We can estimate the integral of the line from ${\sigma_{0}} - iT$ to ${\sigma_{0}} + iT$ by replacing it with the circular path on the right of the line segment connecting ${\sigma_{0}} - iT$ and ${\sigma_{0}} + iT$. Therefore we have
    \begin{equation}\label{error_2}
         \Bigg|\frac{1}{2\pi i}\int_{{\sigma_{0}}-iT}^{{\sigma_{0}}+iT} \frac{x^s}{s^2} ds\Bigg| \leq \frac{1}{2\pi}\pi T \cdot \frac{x^{\sigma_{0}}}{T^2} < \frac{x^{\sigma_{0}}}{T}.
    \end{equation}
    For $x>1$, we take a rectangular contour on the left of the line segment connecting ${\sigma_{0}} - iT$ and ${\sigma_{0}} + iT$ and include the residue at $s=0$, which is $\log x$. For this, we take $U$ such that $U<0<{\sigma_{0}}$, and consider the contour given by the vertices ${\sigma_{0}}-iT$, ${\sigma_{0}}+iT$, $U+iT$, and $U-iT$. The horizontal and vertical integrals of this contour turn out to be the same as in the case of $0<x<1$. Since from the two contours we get the two residues as $0$ and $\log x$, we define the piecewise function
    \begin{equation*}\label{quant perron}
            g(x) = \begin{cases}
                0 & 0<x<1\\
                \log x & x>1.
            \end{cases}
        \end{equation*}
        Then we can say from the equations \eqref{0<x<1},\eqref{error_2} and the similar analysis for $x >1$, that
        \begin{equation}\label{final_error}
           \frac{1}{2 \pi i} \int_{{\sigma_{0}}-iT}^{{\sigma_{0}}+iT} \frac{x^s}{s^2} ds = g(x) + O\Biggl(\frac{x^{{\sigma_{0}}}}{T}\min\Biggl(1, \frac{1}{T\big|\log x\big|}\Biggr)\Biggr).
        \end{equation}
    Now as in equation \eqref{quant perron} if we consider $\frac{x}{n}$ instead of $x$ in \eqref{final_error} we get 
    \begin{equation*}
        \frac{1}{2 \pi i} \int_{{\sigma_{0}}-iT}^{{\sigma_{0}}+iT} \Bigl(\frac{x}{n}\Bigr)^s\frac{1}{s^2} ds = g\Bigl(\frac{x}{n}\Bigr) + O\Biggl(\frac{1}{T}\Bigl(\frac{x}{n}\Bigr)^{\sigma_{0}}\min\Biggl(1, \frac{1}{T\big|\log \frac{x}{n}\big|}\Biggr)\Biggr).
    \end{equation*}
    Now by multiplying throughout by $\sum_{n=1}^{\infty} a_n$, followed by exchanging integral and summation we obtain
    \begin{equation*}
         \frac{1}{2 \pi i} \int_{{\sigma_{0}}-iT}^{{\sigma_{0}}+iT} f(s)\frac{x^s}{s^2} ds = \sum_{n\leq x}a_n\log \Bigl(\frac{x}{n}\Bigr) + O\Biggl(\frac{x^{\sigma_{0}}}{T}\sum_{n=1}^{\infty}\frac{|a_n|}{n^{\sigma_{0}}}\min\Biggl(1, \frac{1}{T\big|\log \frac{x}{n}\big|}\Biggr)\Biggr).
    \end{equation*}
   The result follows by rearranging. 
\end{proof}
\subsection{Preliminaries on Dirichlet series}
In this section, we study the various Dirichlet series associated with the arithmetic function ${(-2)}^{\Omega(n)}$. To evaluate the partial sums
\begin{equation*}
    \sum_{n \leq x} (-2)^{\Omega(n)} \log \Bigl(\frac{x}{n}\Bigr),
\end{equation*}
using Perron's formula, we need to find the abscissa of absolute convergence of the associated Dirichlet series. Let us denote by $J(s)$ the Dirichlet series associated with the arithmetic function ${(-2)}^{\Omega(n)}$, that is
\begin{equation*}
    J(s) = \sum_{n=1}^{\infty} \frac{{(-2)}^{\Omega(n)}}{n^s}.
\end{equation*}
Also, the odd term Dirichlet series comes up in our further calculations, so we define it as
\begin{equation*}
    H(s) = \sum_{\substack{n=1\\n \, \text{odd}}}^{\infty} \frac{(-2)^{\Omega(n)}}{n^s}.
\end{equation*}
We can write $H(s)$ as an Euler product as follows
\begin{equation*}
    H(s) = \sum_{\substack{n=1\\n \, \text{odd}}}^{\infty} \frac{(-2)^{\Omega(n)}}{n^s} = \prod_{p>2} {\biggl(1+\frac{2}{p^s}\biggr)}^{-1} = \prod_{p>2} {\biggl(1-\frac{2}{p^s}+\frac{4}{p^{2s}}-...\biggr)}.
\end{equation*}
We define $G(s)$ as follows
\begin{equation*}
    G(s) = H(s)\zeta^2(s).
\end{equation*}
We first prove the following lemma on $G(s)$.
\begin{lemma}\label{convergence}
    We define the Dirichlet series $G(s)$ by
    \begin{equation*}
        G(s) = \sum_{n=1}^{\infty} \frac{g(n)}{n^s} = H(s)\zeta^2(s).
    \end{equation*}
    Then $G(s)$ converges absolutely in the half-plane $\mathfrak{R}(s)> \log2/\log 3$.
\end{lemma}
\begin{proof}
Since $g(n)$ is a convolution of multiplicative functions, it is also multiplicative. Therefore, we have
\begin{equation*}
    \sum_{n=1}^{\infty} \frac{|g(n)|}{n^{\sigma_{0}}} = \prod_{p} \biggl(1+ \frac{|g(p)|}{p^{\sigma_{0}}}+\frac{|g(p^2)|}{p^{2{\sigma_{0}}}}+...\biggr),
\end{equation*}
where $\mathfrak{R}(s) = \sigma_0$.
Now we have
\begin{equation*}
    H(s) = \prod_{p>2} \Biggl(\sum_{k=0}^{\infty} \biggl(\frac{-2}{p^s}\biggr)^{k}\Biggr).
\end{equation*}
Also for $\zeta^2(s)$ we can write
\begin{equation*}
    \zeta^2(s)= (1-2^{1-s}+2^{-2s})^{-1} \prod_{p>2} \sum_{k=0}^{\infty} \biggl(\frac{k+1}{p^s}\biggr)^{k}.
    \end{equation*}
Let
\begin{equation}
     G(s) = (1-2^{1-s}+2^{-2s})^{-1}\prod_{p > 2} \sum_{k = 0}^{\infty} \frac{a(k)}{p^{ks}},
\end{equation}
where
\begin{equation}
     a(k) = \sum_{j = 0}^{k} (-1)^{k-j} 2^{k-j} (j + 1).
\end{equation}
Now,
\begin{equation}\label{ak} 
|a(k)| \leq 2^k \left| \sum_{j = 0}^{k} \frac{(-1)^j (j + 1)}{2^j} \right|.
\end{equation}
Here, the sum
\[ 
A_k := \sum_{j = 0}^{k} \frac{(-1)^j (j + 1)}{2^j}
\]
is an alternating series in which the absolute value of the summand, $(j+1)/2^j$, decreases toward zero. Therefore, $A_k$ is bounded above by its even partial sums, and bounded below by its odd partial sums. In particular, for all $k \geq 2$,
\[
0 = A_1 \leq A_k \leq A_2 = \frac{3}{4}.
\]
Substituting this into \eqref{ak} gives
\begin{equation}
|a(k)| \leq 3 \cdot 2^{k - 2}
\end{equation}
for all $k \geq 2$.\\
Now, $G(s)$ is absolutely convergent if and only if
\begin{equation} \label{prod}
\sum_{p > 2} \sum_{k = 0}^{\infty} \frac{a(k)}{p^{ks}}
\end{equation}
is convergent. So, in what follows we fix ${\sigma_{0}} = \Re(s) > \log 2 / \log 3$ and use the comparison test on \eqref{prod}.
By \eqref{ak} and the fact that $a(0)=0$ and $a(1) = 0$, we have, for all $p > 2$,
\begin{equation}\label{prodsum}
\left| \sum_{k = 0}^{\infty} \frac{a(k)}{p^{ks}} \right| 
\leq \frac{3}{4} \sum_{k = 2}^{\infty} \frac{2^k}{p^{k{\sigma_{0}}}}.
\end{equation}

The right-hand side of \eqref{prodsum} converges as a geometric sum to \( \frac{3}{p^{2{\sigma_{0}}} - 2p^{\sigma_{0}}} \), provided \( p^{\sigma_{0}} > 2 \). Such a condition holds since \( {\sigma_{0}} > \log 2 / \log 3 \).
Thus
\begin{equation}\label{GP}
\left| \sum_{k = 1}^{\infty} \frac{a(k)}{p^{ks}} \right| 
\leq \sum_{p > 2} \frac{3}{p^{2{\sigma_{0}}} - 2p^{\sigma_{0}}}
\leq \sum_{p > 2} \frac{C({\sigma_{0}})}{p^{2{\sigma_{0}}}}
\end{equation}
for some constant \( C({\sigma_{0}}) > 0 \). Now, the right-most sum of \eqref{GP} converges since \( {\sigma_{0}} > 1/2 \). So by the comparison test, \( G(s) \) is absolutely convergent, as required.
\end{proof}
To compute the integral in Perron's formula, which is
\begin{equation*}
    \frac{1}{2 \pi i} \int_{{\sigma_{0}} - iT}^{{\sigma_{0}} + iT} J(s) \frac{x^s}{s} \text{d}s,
\end{equation*}
we need to consider a suitable contour and apply Cauchy's residue theorem. For this, we find the singularities of $J(s)$ to the left of the line $\mathfrak{R}(s)={\sigma_{0}}$. Since $J(s)$ is convergent in the half-plane $Re(s)>1$, we can take ${\sigma_{0}}>1$. We take ${\sigma_{0}} = 1+\frac{1}{\log x}$. We prove the following lemma on the singularities of $J(s)$.
\begin{lemma}\label{res}
    Define the Dirichlet series $J(s)$ as
    \begin{equation*}
        J(s) = \sum_{n=1}^{\infty} \frac{{(-2)}^{\Omega(n)}}{n^s}
    \end{equation*}
    for $\mathfrak{R}(s)>1$. Let
    \begin{equation*}
        \sigma_0 = 1 + \frac{1}{\log x}.
    \end{equation*}
    The singularities of $J(s)$ to the left of the vertical line connecting ${\sigma_{0}} + iT$ and ${\sigma_{0}} -iT$ are given by
    \begin{equation*}
        z_k = 1 - \frac{i(2k+1)\pi}{\log 2}.
    \end{equation*}
    Additionally, we have the residue
    \begin{equation*}
        \Res_{s=z_k} \frac{J(s)x^s}{s^2} = \frac{H(z_k)x^{z_k}}{{z_k}^2}
    \end{equation*}
    at $s=z_k$ where
    \begin{equation*}
        H(s) = \sum_{\substack{n=1\\ n \hspace{2 pt} \text{odd}}}^{\infty} \frac{(-2)^{\Omega(n)}}{n^s}.
    \end{equation*}
\end{lemma}
\begin{proof}
    We write $J(s)$ in Euler product form as
\begin{equation*}
    J(s) = \sum_{n=1}^{\infty} \frac{{(-2)}^{\Omega(n)}}{n^s} = \prod_{p}^{} \biggl(1+ \frac{2}{p^s} \biggr)^{-1}
\end{equation*}
From this Euler product, we can compute the singularities of $J(s)$ to the left of the line $\mathfrak{R}(s)={\sigma_{0}}$. Since ${\sigma_{0}}=1+\frac{1}{\log x}$, we consider the singularities of $J(s)$ having $\mathfrak{R}(s)=1$. Setting $\Bigl(1+ \frac{2}{p^s}\Bigr)$ to 0, we get
\begin{equation*}
    s = \frac{\log 2-i(2k+1)\pi}{\log p}
\end{equation*}
Since we are interested in the singularities to the left of the line $\mathfrak{R}(s)={\sigma_{0}}$, we only consider the singularities
\begin{equation*}
    z_k = 1 - \frac{i(2k+1)\pi}{\log 2}.
\end{equation*}
Now we compute the residue at these singularities.
\begin{align}
\lim_{z \to z_k} \frac{(z-z_k)}{{(1+2^{1-z})}} \prod_{p>2} {\biggl(1+\frac{2}{p^z}\biggr)}^{-1}\frac{x^z}{z^2}\nonumber &= \lim_{z \to z_k} \frac{1}{{(-2^{1-z}\log 2)}} \prod_{p>2} {\biggl(1+\frac{2}{p^z}\biggr)}^{-1}\frac{x^z}{z^2}\nonumber\\
&= \frac{H(z_k)x^{z_k}}{-{z_k}^2 2^{1-z_k}\log2}
\end{align}
Now, in the denominator we observe that,
\begin{align*}
   1-z_k &= 1-\biggl(1-\frac{(2k+1)\pi i}{\log2}\biggr)\\
   &= \frac{(2k+1)\pi i}{\log2}.
\end{align*}
Therefore,
\begin{align*}
    2^{1-z_k} &= {\Bigl(2^{\frac{1}{\log2}}\Bigr)}^{(2k+1)\pi i}\\
    &= e^{(2k+1)\pi i}\\
    &=-1.
\end{align*}
So we can rewrite the residue as
\begin{equation*}
      \Res_{s=z_k} J(s) = \frac{H(z_k)x^{z_k}}{{z_k}^2}
\end{equation*}
\end{proof}
Since we can write the odd term Dirichlet series $H(s)$ in terms of the reciprocal of $\zeta(s)$, it is important to note the zero free regions of $\zeta(s)$ so that we take the contour integral in a region which does not contain any zeros of $\zeta(s)$. The methods of Korobov \cite{Kor58} and Vinogradov \cite{Vin58} produce a zero free region for the Riemann zeta function $\zeta(s)$ of the following strength.
\begin{lemma} \label{VK ZFR}
For some constant $c > 0$, there are no zeros of $\zeta(s)$ for $s = \beta + it$ with $|t|$ large and
\[
1 - \beta \leq \frac{c}{(\log |t|)^{2/3} \, (\log \log|t|)^{1/3}}.
\]
\end{lemma}
The essential ingredient for this Lemma \ref{VK ZFR} is the following upper bound on $\frac{1}{\zeta(s)}$.
\begin{lemma}\label{zeta}
Let $s = {\sigma_{0}} + it$ and given $1-c'(\log t)^{-\frac{1}{3}}(\log\log t)^{-\frac{2}{3}} \leq {\sigma_{0}}$ for some absolute constant $c'$,
\begin{equation}
    \frac{1}{\zeta(s)} \ll (\log t)^{\frac{2}{3}}(\log \log t)^{\frac{1}{3}}.
\end{equation}
\end{lemma}
\subsection{Bounds for $G(s)$ and $\frac{1}{\zeta^2(s)}$}\label{G(s)}
While applying the log-weighted Perron formula Lemma \ref{perron}, we need to evaluate the contour integral on the line connecting ${\sigma_{0}} + iT$ and ${\sigma_{0}} -iT$. To do this, we use Cauchy's residue theorem. The resulting residue term contains the odd term Dirichlet series $H(s)$ given by
\begin{equation*}
    H(s) = \sum_{\substack{n=1\\n \, \text{odd}}}^{\infty} \frac{(-2)^{\Omega(n)}}{n^s} = \frac{G(s)}{\zeta^2(s)}.
\end{equation*}
This is evaluated at the singularities of $J(s)$. We recall that the Dirichlet series
    \begin{equation*}
        J(s) = \sum_{n=1}^{\infty} \frac{{(-2)}^{\Omega(n)}}{n^s},
    \end{equation*}
    has singularities to the left of the $1$-line which are given by
    $$z_k = 1-\frac{(2k+1)\pi i}{\log 2}.$$
    We can estimate $|G(z_k)|$ by computation using the Euler product of $G(s)$. For $\Big|\frac{1}{\zeta^2(z_k)}\Big|$ we apply an existing bound on $\frac{1}{\zeta(s)}$ on the $1$-line. We discuss these bounds below. In \cite{Johnston2024NewBA}, Proposition 2.3, we get the following bound on $G(z_k)$.
\begin{lemma}\label{G}
    We have
    \begin{equation*}
        |G(z_k)| \leq 0.722925.
    \end{equation*}
\end{lemma}
In \cite{leong2024explicitestimateslogarithmicderivative}, Leong proved the following bounds on the reciprocal of $\zeta(s)$ on the $1$-line.
\begin{lemma}\label{reciprocalzeta}
For $2\leq t \leq 500$ we have
\begin{equation*}
        \Bigg|\frac{1}{\zeta(1+it)}\Bigg| \leq 2.079\log t,
\end{equation*}
and for $500\leq t \leq 2\exp(e^2)$ we have
    \begin{equation*}
        \Bigg|\frac{1}{\zeta(1+it)}\Bigg| \leq 1.288 \log t,
    \end{equation*}
and for $t \geq 2\exp(e^2)$, we have
\begin{equation*}
    \Bigg|\frac{1}{\zeta(1+it)}\Bigg| \leq 29.388 \log t.
\end{equation*}
\end{lemma}

\section{Proof of Theorem 1.1}\label{proof section}
Applying the log-weighted Perron formula Lemma \ref{perron lemma} to the Dirichlet series $J(s)$, we get
\begin{align*}
    \sum_{n\leq x}^{} {(-2)}^{\Omega(n)} \log\Bigl(\frac{x}{n}\Bigr)={} &\frac{1}{2 \pi i} \int_{{\sigma_{0}} - iT}^{{\sigma_{0}} + iT} J(s) \frac{x^s}{s} \text{d}s\\ &+ O\Biggl(\frac{1}{T}\sum_{n=1}^{\infty} {\Bigl(\frac{x}{n}\Bigr)}^{{\sigma_{0}}} {2}^{\Omega(n)} \min\biggl(1, \frac{1}{T\big|\log \frac{x}{n}\big|}\biggr)\Biggr),
\end{align*}
where $T = (\log x)^{1+\delta}$. Also we recall that ${\sigma_{0}} = 1+\frac{1}{\log x}$.
\subsection{Perron's formula}
We evaluate the main integral term of Perron's formula using the residue theorem. From Lemma \ref{res}, it can be written as
\begin{align}
        \frac{1}{2\pi i} \int_{{\sigma_{0}}-iT}^{{\sigma_{0}}+iT} \frac{J(s)x^s}{s^2}\,\text{d}s &= \sum_{|z_k|<T} \frac{H(z_k)x^{z_k}}{{z_k}^2}\nonumber\\ &- \frac{1}{2 \pi i} \biggl\{ \biggl(\int_{{\sigma_{0}}+iT}^{d+iT}+\int_{d+iT}^{d-iT}+\int_{d-iT}^{{\sigma_{0}}-iT}\biggr) \frac{J(s)x^s}{s^2}\,\text{d}s\biggr\}
\end{align}
where, ${\sigma_{0}} = 1+\frac{1}{\log x}$ and $d = 1-\frac{c'}{(\log T)^{2/3}(\log\log T)^{1/3}}$ for some absolute constant $c'$.\\
\subsection{Other integrals}
We can now evaluate the other integrals using the bounds for $\zeta^2(s)$ as follows.
\begin{align*}
    \Bigg|\frac{1}{2 \pi i}\int_{{\sigma_{0}}+iT}^{d+iT} \frac{J(s)x^s}{s^2}\,\text{d}s \Bigg| &= \Bigg|\frac{1}{2 \pi i}\int_{{\sigma_{0}}+iT}^{d+iT} \frac{G(s)x^s}{(1+2^{1-s})s^2\zeta^2(s)}\,\text{d}s \Bigg|\\
    &\ll \frac{x^{\sigma_{0}}(\log T)^{\frac43}(\log \log T)^\frac23}{T^2}.
\end{align*}
Similarly we can write,
\begin{align*}
    \Bigg|\frac{1}{2 \pi i}\int_{d-iT}^{{\sigma_{0}}-iT} \frac{J(s)x^s}{s^2}\,\text{d}s \Bigg| &= \Bigg|\frac{1}{2 \pi i}\int_{d-iT}^{{\sigma_{0}}-iT} \frac{G(s)x^s}{(1+2^{1-s})s^2\zeta^2(s)}\,\text{d}s \Bigg|\\
    &\ll \frac{x^{\sigma_{0}}(\log T)^{\frac43}(\log \log T)^\frac23}{T^2}.
\end{align*}
Now,
\begin{align*}
    \Bigg|\frac{1}{2 \pi i}\int_{d+iT}^{d-iT} \frac{J(s)x^s}{s^2}\,\text{d}s \Bigg| &= \Bigg|\frac{1}{2 \pi i}\int_{d+iT}^{d-iT} \frac{G(s)x^s}{(1+2^{1-s})s^2\zeta^2(s)}\,\text{d}s \Bigg|\\
    &\ll x^d (\log T)^{\frac43}(\log \log T)^\frac23\int_{T}^{-T}\frac{1}{d^2+t^2}\text{d}t
\end{align*}
Now since,
\begin{equation*}
\int_{T}^{-T} \frac{1}{d^2+t^2}\,\text{d}t = O(1),
\end{equation*}
we can write,
\begin{equation*}
\Bigg|\frac{1}{2 \pi i}\int_{d+iT}^{d-iT} \frac{J(s)x^s}{s^2}\,\text{d}s \Bigg| \ll x^d(\log T)^{\frac43}(\log \log T)^\frac23.
\end{equation*}
\subsection{Error term}
We estimate the error term using the quantitative Perron's formula as derived earlier in \eqref{perron}. The estimate derived here is of the order of $o(x)$ is sufficient for our purposes as the main term is of order $x$.
\begin{theorem}
    The error term $E(x)$ in the main Perron's formula is
    \begin{equation*}
        E(x) = o(x).
    \end{equation*}
\end{theorem}
\begin{proof}
    The error term is 
    \begin{equation*}
        \frac{1}{T^2} \sum_{n=1}^{\infty} {\Bigl(\frac{x}{n}\Bigr)}^{\sigma_0}\frac{2^{\Omega(n)}}{\big|\log \frac{x}{n}\big|}.
    \end{equation*}
    We can split this error term into three parts depending on the distance of $n$ from $x$, since $\big|\log \frac{x}{n}\big|^{-1}$ is bounded when $n$ is far from $x$. We split it into the following three parts
    \begin{equation*}
        \frac{1}{T^2} \sum_{n=1}^{\infty} {\Bigl(\frac{x}{n}\Bigr)}^{{\sigma_{0}}}\frac{2^{\Omega(n)}}{\big|\log \frac{x}{n}\big|} = \frac{x^{{\sigma_{0}}}}{T^2} \Biggl\{\sum_{n \leq \frac{x}{2}} + \sum_{\frac{x}{2} < n \leq \frac{3x}{2}} + \sum_{n > \frac{3x}{2}}\Biggr\} \frac{2^{\Omega(n)}}{n^{{\sigma_{0}}}\big|\log \frac{x}{n}\big|}.
    \end{equation*}
    For $n \leq \frac{x}{2}$ and $n > \frac{3x}{2}$,
    \begin{equation*}
        \bigg|\log \frac{x}{n}\bigg| \geq \log2 \gg 1.
    \end{equation*}
    Therefore, we can bound the first and the third sum by
    \begin{equation*}
        \frac{x^{{\sigma_{0}}}}{T^2} \Biggl\{\sum_{n \leq \frac{x}{2}} + \sum_{n > \frac{3x}{2}}\Biggr\} \frac{2^{\Omega(n)}}{n^{{\sigma_{0}}}\big|\log \frac{x}{n}\big|} \ll \frac{x^{{\sigma_{0}}}}{T^2} f({\sigma_{0}}),
    \end{equation*}
    where
    \begin{equation*}
        f(s) = \sum_{n=1}^{\infty} \frac{2^{\Omega(n)}}{n^s}.
    \end{equation*}
    Now we can write $f({\sigma_{0}})$ as
    \begin{equation*}
        f({\sigma_{0}}) = \zeta^2({\sigma_{0}}) \prod_{p}(1+(p^{2{\sigma_{0}}}-2p^{\sigma_{0}})^{-1})^{-1}.
    \end{equation*} 
    Therefore, the first and the third term gives us an error of
    \begin{equation*}
        \frac{x^{{\sigma_{0}}}}{T^2} f({\sigma_{0}}) \ll \frac{x^{{\sigma_{0}}}}{T^2}\zeta^2({\sigma_{0}}) = O\biggl(\frac{x^{{\sigma_{0}}}}{T^2({\sigma_{0}}-1)^2}\biggr) = O\biggl(\frac{x\log^2x}{T^2}\biggr).
    \end{equation*}
    For the middle term where $\frac{x}{2}<n\leq \frac{3x}{2}$, $\bigl(\frac{x}{n}\bigr)^{{\sigma_{0}}}$ is bounded, so we have
    \begin{equation*}
        \frac{1}{T^2} \sum_{\frac{x}{2}<n\leq \frac{3x}{2}} \frac{x^{{\sigma_{0}}}2^{\Omega(n)}}{n^{{\sigma_{0}}}\big|\log \frac{x}{n}\big|} \ll \frac{x}{T^2} \sum_{\frac{x}{2}<n\leq \frac{3x}{2}} \frac{2^{\Omega(n)}}{|x-n|}. 
    \end{equation*}
    We split this sum again to three parts as
    \begin{equation}\label{error term split}
        \frac{x}{T^2} \sum_{\frac{x}{2}<n\leq \frac{3x}{2}} \frac{2^{\Omega(n)}}{|x-n|} = \frac{x}{T^2} \Bigg(\sum_{\frac{x}{2}<n\leq x-\frac{1}{2}} \frac{2^{\Omega(n)}}{|x-n|} + \sum_{x+\frac{1}{2}<n\leq \frac{3x}{2}} \frac{2^{\Omega(n)}}{|x-n|}  + 2^{\Omega(x +\frac12) +1}\Bigg).
    \end{equation}
    We apply partial summation to the first term on the right hand side to get
    \begin{align*}
        \frac{x}{T^2} \Bigg(\sum_{\frac{x}{2}<n\leq x-\frac{1}{2}} \frac{2^{\Omega(n)}}{|x-n|}\Bigg) = \frac{x}{T^2} \Bigg(\sum_{n\leq x-\frac{1}{2}} 2^{\Omega(n)+1} &-\frac1x\sum_{n\leq \frac{x}{2}} 2^{\Omega(n)+1}\\ &+ \int_{\frac{x}{2}}^{x-\frac{1}{2}} \sum_{n \leq u}^{} 2^{\Omega(n)} \cdot \frac{1}{(u-x)^2}\text{d}u \Bigg).
    \end{align*}
    Now applying Proposition \ref{grosswald} to the last term on the right hand side we get
    \begin{align}
        \frac{x}{T^2} \Bigg(\sum_{\frac{x}{2}<n\leq x-\frac{1}{2}} \frac{2^{\Omega(n)}}{|x-n|}\Bigg) = \frac{x}{T^2} &\Bigg(\sum_{n\leq x-\frac{1}{2}} 2^{\Omega(n)+1} -\frac1x\sum_{n\leq \frac{x}{2}} 2^{\Omega(n)+1}\nonumber \\ &+ O\biggl(\int_{\frac{x}{2}}^{x-\frac{1}{2}} u \log^2\biggl(\frac{3u}{2}\biggr)\cdot \frac{1}{(u-x)^2}\text{d}u \biggr) \Bigg).\label{term 1}
    \end{align}
    Similarly, we can evaluate the second term in \eqref{error term split}.
    \begin{align}
        \frac{x}{T^2} \Bigg( \sum_{x+\frac{1}{2}<n\leq \frac{3x}{2}} \frac{2^{\Omega(n)}}{|x-n|}\Bigg) = \frac{x}{T^2} &\Bigg( \frac1x\sum_{n\leq \frac{3x}{2}} 2^{\Omega(n)+1} - \sum_{n\leq x+\frac{1}{2}} 2^{\Omega(n)+1}\nonumber \\ &+ O\biggl(\int_{x +\frac12}^{\frac{3x}{2}} u \log^2\biggl(\frac{3u}{2}\biggr) \cdot \frac{1}{(u-x)^2} \text{d}u\biggr) \Bigg).\label{term 2}
    \end{align}
    Now we replace \eqref{term 1} and \eqref{term 2} in \eqref{error term split} to get
    \begin{align*}
        \frac{x}{T^2} \sum_{\frac{x}{2}<n\leq \frac{3x}{2}} \frac{2^{\Omega(n)}}{|x-n|} &= \sum_{n\leq x-\frac{1}{2}} 2^{\Omega(n)+1} -\frac1x\sum_{n\leq \frac{x}{2}} 2^{\Omega(n)+1} - \sum_{n\leq x+\frac{1}{2}} 2^{\Omega(n)+1}\\ &+ \frac1x\sum_{n\leq \frac{3x}{2}} 2^{\Omega(n)+1} + 2^{\Omega(x +\frac12) +1}\\ &+ O\biggl(\int_{x +\frac12}^{\frac{3x}{2}} u \log^2\biggl(\frac{3u}{2}\biggr) \cdot \frac{1}{(u-x)^2} \text{d}u\biggr)\\ &+ O\biggl(\int_{\frac{x}{2}}^{x-\frac{1}{2}} u \log^2\biggl(\frac{3u}{2}\biggr)\cdot \frac{1}{(u-x)^2}\text{d}u \biggr).
    \end{align*}
    Rearranging we get
    \begin{align*}
        \frac{x}{T^2} \sum_{\frac{x}{2}<n\leq \frac{3x}{2}} \frac{2^{\Omega(n)}}{|x-n|} &= \frac{x}{T^2} \Bigg(2^{\Omega(x +\frac12) +1} - \Bigg(\sum_{n\leq x+\frac{1}{2}} 2^{\Omega(n)+1} - \sum_{n\leq x-\frac{1}{2}} 2^{\Omega(n)+1} \Bigg)\\
        & + \frac1x\Bigg(\sum_{n\leq \frac{3x}{2}} 2^{\Omega(n)+1} - \sum_{n\leq \frac{x}{2}} 2^{\Omega(n)+1}\Bigg)\\
        &+ O\biggl(\int_{x +\frac12}^{\frac{3x}{2}} u \log^2\biggl(\frac{3u}{2}\biggr) \cdot \frac{1}{(u-x)^2} \text{d}u\biggr)\\ &+ O\biggl(\int_{\frac{x}{2}}^{x-\frac{1}{2}} u \log^2\biggl(\frac{3u}{2}\biggr)\cdot \frac{1}{(u-x)^2}\text{d}u \biggr)\Bigg).
    \end{align*}
    Now, in the right hand side, the first three terms cancel and applying Proposition \ref{grosswald} to the fourth and fifth term gives $O\biggl(\frac{x\log^2(x)}{T^2}\biggr)$. Upon evaluating the integrals we get an error term of
    \begin{equation*}
        O\biggl(\frac{x\log^2x}{T^2}\biggr).
    \end{equation*}
    Therefore the error terms are
    \begin{equation*}
        O\biggl(\frac{x(\log T)^{\frac43}(\log \log T)^{\frac23}}{T^2}\biggr) +O\biggl(\frac{x\log^2x}{T^2}\biggr)+O\biggl(\frac{x(\log T)^{\frac43}(\log \log T)^{\frac23}}{x^{c'(\log T)^{-\frac23}(\log\log T)^{-\frac13}}}\biggr).
    \end{equation*}
    To make this $o(x)$ we can take $T = (\log x)^{1+\delta}$. For the third term we have to show\
    \begin{equation*}
        \lim_{x \to \infty} \frac{{x^{c'(\log T)^{-\frac23}(\log\log T)^{-\frac13}}}}{(\log T)^{\frac43}(\log \log T)^{\frac23}}= \infty.
    \end{equation*}
    The numerator can be rewritten as
    \begin{equation*}
        \lim_{x \to \infty}\exp\biggl(\frac{c'\log x}{(\log T)^{\frac23}(\log\log T)^{\frac13}}\biggr) = \infty.
    \end{equation*}
    Since the exponential function grows faster than the logarithm, this term goes to $0$ as $x \to \infty$.\\

    For the first term, since the polynomial term grows faster than the logarithms, it is $o(x)$. For the second term,
    \begin{equation*}
        O\biggl(\frac{x\log^2x}{T^2}\biggr) = O\biggl(\frac{x}{(\log x)^{2\delta}}\biggr).
    \end{equation*}
\end{proof}
\section{Proof of Theorem \ref{main theorem}}
By definition we can write
\begin{equation}\label{H and G}
    H(s) = \frac{G(s)}{\zeta^2(s)}.
\end{equation}
Here $G(s)$ converges absolutely in the half-plane $\mathfrak{R}(s)>\frac{\log 2}{\log 3}$.
Therefore, by putting \eqref{H and G} in Lemma \ref{res} we get the residue term 
\begin{equation*}
   \sum_{|\mathfrak{Im}(z_k)|<T} \frac{G(z_k)x^{z_k}}{{z_k}^2 \zeta^2(z_k)}.
\end{equation*}
\subsection{Computing the main term}
We can bound the residue term M as
\begin{align*}
    M=\Bigg|\frac{1}{\log2}\sum_{|\mathfrak{Im}(z_k)|<T} \frac{G(z_k)x^{z_k}}{\zeta^2(z_k){z_k}^2}\Bigg| &\leq \frac{x}{\log2} \sum_{|\mathfrak{Im}(z_k)|<T} \frac{|G(z_k)|}{{|\zeta(z_k)|}^2{|z_k|}^2}\\
    &\leq \frac{x}{\log2} \sum_{k \in \mathbb{Z}} \frac{|G(z_k)|}{{|\zeta(z_k)|}^2{|z_k|}^2}.
\end{align*}
Since for any integer $k$ and complex number $s$, the symmetry
\begin{equation*}
    \overline{k^{s}} = k^{\overline{s}}
\end{equation*}
holds, we can write
\begin{equation*}
    \overline{\zeta(s)} = \zeta(\overline{s}).
\end{equation*}
This implies that
\begin{equation*}
    \overline{\zeta\biggl({1-\frac{(2k+1)\pi i}{\log 2}}\biggr)} = \zeta\biggl({1+\frac{(2k+1)\pi i}{\log 2}}\biggr).
\end{equation*}
Therefore, we have
\begin{equation*}
    {\Bigg|\overline{\zeta\biggl({1-\frac{(2k+1)\pi i}{\log 2}}\biggr)}\Bigg|}^2 = {\Bigg|\zeta\biggl({1+\frac{(2k+1)\pi i}{\log 2}}\biggr)\Bigg|}^2.
\end{equation*}
Also,
\begin{equation*}
   {\Bigg|1+\frac{(2k+1)\pi i}{\log 2}\Bigg|}^2 = {\Bigg|1-\frac{(2k+1)\pi i}{\log 2}\Bigg|}^2.
\end{equation*}
Therefore we rewrite the sum as
\begin{equation*}
    M \leq \frac{2x}{\log2} \sum_{k=0}^{\infty} \frac{|G(z_k)|}{{|\zeta(z_k)|}^2{|z_k|}^2}.
\end{equation*}
From \ref{G} we know that
\begin{equation*}
    |G(z_k)| \leq 0.722925.
\end{equation*}
Therefore we need to compute
\begin{equation*}
    M \leq 2.08592x\sum_{k=0}^{54} \frac{1}{{|\zeta(z_k)|}^2{|z_k|}^2} + 2.08592x\sum_{k=55}^{\infty} \frac{1}{{|\zeta(z_k)|}^2{|z_k|}^2}.
\end{equation*}
The first term in the above sum can be computed using an existing database of $\zeta(s)$ values on the $1-$line. The code is included in the Appendix (Section \ref{zeta_compute}). Therefore, we get
\begin{equation*}
    M \leq 1.3210x +2.0859x\sum_{k=55}^{\infty} \frac{1}{{|\zeta(z_k)|}^2{|z_k|}^2}.
\end{equation*}
Now we apply the bounds on $\frac{1}{\zeta(z_k)}$ from \ref{reciprocalzeta} and get
\begin{equation}\label{M}
    M \leq 1.3210x + 2.6866x\sum_{k=55}^{1120} \frac{\log^2\Bigl(\frac{(2k+1)\pi}{\log2}\Bigr)}{1+\frac{\pi^2{(2k+1)}^2}{\log^22}}+61.3004x\sum_{k=1121}^{\infty} \frac{\log^2\Bigl(\frac{(2k+1)\pi}{\log2}\Bigr)}{1+\frac{\pi^2{(2k+1)}^2}{\log^22}}.
\end{equation}
For the second and third terms of this sum, let us denote them by
\begin{equation*}
S =  \sum_{k=55}^{1120} \frac{\log^2\Bigl(\frac{(2k+1)\pi}{\log2}\Bigr)}{1+\frac{\pi^2{(2k+1)}^2}{\log^22}} \quad \text{and} \quad T = \sum_{k=1121}^{\infty} \frac{\log^2\Bigl(\frac{(2k+1)\pi}{\log2}\Bigr)}{1+\frac{\pi^2{(2k+1)}^2}{\log^22}}.
\end{equation*}
To estimate S and T, we require the following inequalities. For $x\geq48$, we have
\begin{equation*}
\log^2(x) \leq x^{0.7}.
\end{equation*}
Also, for $x \geq 5503$ we have
\begin{equation*}
    \log^2(x) \leq x^{0.5}.
\end{equation*}
So we can write
\begin{equation*}
    S =  \sum_{k=55}^{1120} \frac{\log^2\Bigl(\frac{(2k+1)\pi}{\log2}\Bigr)}{1+\frac{\pi^2{(2k+1)}^2}{\log^22}} \leq \sum_{k=55}^{1120} \frac{{\Bigl(\frac{(2k+1)\pi}{\log2}\Bigr)}^{0.7}}{{\Bigl(\frac{(2k+1)\pi}{\log2}\Bigr)}^2} \leq  \sum_{k=55}^{1120} {\biggl(\frac{(2k+1)\pi}{\log2}\biggr)}^{-1.3}. 
\end{equation*}
Now we use the integral inequality
\begin{equation*}
\sum _{n=N}^{\infty }f(n)\leq f(N)+\int _{N}^{\infty }f(x)\,\text{d}x.
\end{equation*}
Now applying this to $S$ we get
\begin{align*}
    S \leq \sum_{k=55}^{1120} {\biggl(\frac{(2k+1)\pi}{\log2}\biggr)}^{-1.3} \leq 0.0004 + \int_{55}^{1120} {\biggl(\frac{(2x+1)\pi}{\log2}\biggr)}^{-1.3} \text{d}x
\end{align*}
Finally we get
\begin{equation*}
    S \leq 0.0342.
\end{equation*}
Similarly we have
\begin{equation*}
    T \leq 0.0022.
\end{equation*}
Applying the estimates for $S$ and $T$ in \eqref{M} we get,
\begin{align*}
    M &\leq 1.3210x + 0.0919x + 0.1349x\\
    &= 1.5478x.
\end{align*}
\section*{Acknowledgements}
The author would like to thank Bryce Kerr for suggesting this problem and for
his guidance while working on it and writing the current paper. The author would also
like to thank Wasim Reza for his guidance on using the Katana computing cluster.
\appendix
\section{Computing residue term}\label{zeta_compute}
In this appendix, we discuss the computation of a part of the main residue term from a database of existing $\zeta(s)$-values. We use the following Python program.This research includes computations using the computational cluster Katana supported by Research Technology Services at UNSW Sydney \cite{unsw_katana_latex}.
\begin{lstlisting}[language=Python]
import math
from mpmath import *
s=0
for i in range(0,55,1):
    z=zeta(1+ (j*((2*i)+1))/(math.log(2)))
    a=abs(z)
    k = 1/((a**2)*(1+ ((2*i+1)**2)/((math.log(2))**2)))
    s=s+k
m=2.08592*s
print (m)
    \end{lstlisting}
\bibliographystyle{plain}
\bibliography{references.bib}\label{refs}
\end{document}